\documentclass[leqno]{aims}

\usepackage{amsmath, amsfonts, amssymb}
\usepackage{paralist}
\usepackage{graphics}
\usepackage{epsfig}
 \usepackage[colorlinks=true]{hyperref}
\hypersetup{urlcolor=blue, citecolor=red}

\def\currentvolume{x}
 \def\currentissue{x}
  \def\currentyear{2014}
   \def\currentmonth{xxx}
    \def\ppages{1--xx}

\newtheorem{theorem}{Theorem}[section]

\newtheorem{lemma}[theorem]{Lemma}

\theoremstyle{definition}

\newtheorem{remark}{Remark}[section]

\numberwithin{equation}{section}

\def\bfR{\mathbb{R}}

\def\mcK{\mathcal{K}}
\def\mcB{\mathcal{B}}

\title[Global Solvability for 1.5D VLASOV MAXWELL]{Separated Characteristics and Global Solvability for the one and one-half dimensional Vlasov Maxwell System}

\author[Robert Glassey \\
Stephen Pankavich\\
Jack Schaeffer]{}

\subjclass{Primary: 35L60, 35Q83, 35Q99; Secondary: 82C21, 82C22, 82D10}
 \keywords{Kinetic Theory, Vlasov-Maxwell, global existence}

 \email{glassey@indiana.edu} 
 \email{pankavic@mines.edu}
 \email{js5m@andrew.cmu.edu}

 \thanks{This work was supported in part by the National Science Foundation under Grant DMS-1211667.}

\begin{document}

\maketitle

\centerline{\scshape Robert Glassey}
\medskip
{\footnotesize
 \centerline{Department of Mathematics}
 \centerline{Indiana University}
 \centerline{Bloomington, IN  47405 USA}
}

\bigskip

\centerline{\scshape Stephen Pankavich}
\medskip
{\footnotesize
 \centerline{Department of Applied Mathematics and Statistics}
 \centerline{Colorado School of Mines}
 \centerline{Golden, CO 80002 USA}
} %

\bigskip

\centerline{\scshape Jack Schaeffer}
\medskip
{\footnotesize
 \centerline{Department of Mathematical Sciences}
 \centerline{Carnegie Mellon University}
 \centerline{Pittsburgh, PA 15213 USA}

} 

\medskip

\begin{abstract}
The motion of a collisionless plasma - a high-temperature, low-density, ionized gas - is described by the Vlasov-Maxwell (VM) system. These equations are considered in one space dimension and two momentum dimensions without the assumption of relativistic velocity corrections. The main results are bounds on the spatial and velocity supports of the particle distribution function and uniform estimates on derivatives of this function away from the critical velocity $\vert v_1 \vert = 1$.  Additionally, for initial particle distributions that are even in the second velocity argument $v_2$, the global-in-time existence of solutions is shown.  
\end{abstract}

\section{Introduction}
A plasma is a partially or completely ionized gas.  When a plasma is of low density or the time scales of interest are sufficiently small, it is deemed to be ``collisionless'', as collisions between particles become infrequent.  
The fundamental equations which describe the time evolution of a collisionless plasma are given by the three-dimensional Vlasov-Maxwell system:
\begin{equation} \label{VM} \tag{VM} \left \{ \begin{gathered} \partial_t f + v \cdot \nabla_x f + \left (E + v \times B \right ) \cdot \nabla_v f = 0 \\  \rho(t,x) = \int f(t,x,v) \ dv, \quad j(t,x)= \int v  f(t,x,v) \ dv \\  \partial_t E = \nabla \times B - j, \qquad \nabla \cdot E = \rho \\  \partial_t B = - \nabla \times E,
\qquad \nabla \cdot B = 0. \\ \end{gathered} \right.
\end{equation}
This nonlinear system of integro-differential equations is supplemented by a set of initial conditions
$f(0,x,v) = f^0(x,v), E(0,x) = E^0(x)$, and $B(0,x) = B^0(x).$
Here, $f$ represents the density of (positively-charged) ions in the plasma, while $\rho$ and $j$ are the charge and current densities, and $E$ and $B$ represent electric and magnetic fields generated by the charge and current.  The independent variables, $t \geq 0$ and $x,v \in \bfR^3$ represent time, position, and velocity, respectively, and physical constants, such as 
the speed of light $c$, have been normalized.  In the presence of large velocities, relativistic corrections may be necessary. The corresponding system to consider is then the relativistic analogue of (\ref{VM}), denoted by (RVM) and constructed by replacing $v$ with $\hat{v} = \frac{v}{\sqrt{1 + \vert v \vert^2}}$ within the first equation of (\ref{VM}), called the Vlasov equation, and the integrand of the current $j$.  
For a general reference concerning kinetic models of plasma dynamics, such as \eqref{VM} and (RVM), see \cite{Glassey, VKF}.  

Over the years some progress has been made in the analysis of (RVM), specifically establishing the global existence of weak solutions (which also holds for (VM); see \cite{DPL}) and determining a sufficient condition which ensures global existence of classical solutions 
for the Cauchy problem \cite{GlStr}.
In lower-dimensional settings, this condition has been shown to hold \emph{a priori} \cite{GlaSch90, GlaSch97, GlaSch981, GlaSch982}. 
Additionally, a wide array of information has been discovered regarding the electrostatic versions of both (\ref{VM}) and (RVM), known as the Vlasov-Poisson and relativistic Vlasov-Poisson systems, respectively.  These models do not include magnetic effects
, and the electric field is given by an elliptic equation, rather than a hyperbolic system of PDEs. This simplification has led to a great deal of progress concerning the electrostatic systems, including theorems regarding global existence and long-time behavior of solutions 
\cite{KRM, MMAS, LP, Pfaff, Jack}. However, a global existence theorem for classical solutions stemming from arbitrary data in the relativistic case has remained elusive.  Independent of these advances, many of the most basic 
well-posedness questions remain unsolved for (\ref{VM}), and few results exist within the literature, with \cite{GPSRVM, GlaStr} representing exceptions.  The main difficulty which arises is the loss of strict hyperbolicity of the kinetic system due to the possibility that particle velocities $v$ may travel faster than the propagation of signals from the electric and magnetic fields, which do so at the speed of light $c = 1$ (cf. \cite{PN}). As one can see, this difficulty is remedied physically by the inclusion of relativistic velocity corrections which uniformly constrain velocities $\vert \hat{v} \vert < 1$.  In many macroscopic physical systems one does not consider the effects of special relativity, but at the kinetic level such velocity corrections may play a fundamental role, even in the basic well-posedness of solutions.  Hence, one of the primary goals of the current work is to understand how this 
affects such properties, and establish a precise result that guarantees the continued smoothness of solutions as long as velocity characteristics do not assume magnitudes that approach $c$.

Often a remedy to the lack of progress on such a problem is to reduce the dimensionality of the system.  
The lowest-dimensional reduction which retains magnetic effects is the so-called ``one-and-one-half-dimensional'' Vlasov-Maxwell system which is constructed by taking $x \in \bfR$ but $v \in \bfR^2$, yielding the system of PDEs
\begin{equation} \tag{1.5D VM} \label{1.5DVM} \left \{ \begin{gathered}
 \partial_t f + v_1\partial_x f + K \cdot \nabla_v f = 0\\
 K = \langle E_1 + v_2 B, E_2 - v_1 B \rangle, \qquad j(t,x) = \int v f \ dv\\
 \partial_x E_1 = \rho(t,x) = \int f dv - b(x), \qquad \partial_t E_1 = - j_1 \\
\partial_t E_2 = -\partial_x B - j_2, \qquad
\partial_t B = - \partial_x E_2
\end{gathered} \right.
\end{equation}
and initial conditions
\begin{equation}
\label{IC}
\tag{IC}
f(0,x,v) = f^0(x,v), \qquad E_2(0,x) = E_2^0(x), \qquad B(0,x) = B^0(x).
\end{equation}
Here, $B$ is now a scalar field and the associated electric field possesses only two components, $E_1$ and $E_2$. 
Additionally, the given, neutralizing background density $b$ is included in order to study solutions with finite energy.  If one takes $b \equiv 0$, then $E_1 \not\in L^2(\bfR)$ and solutions necessarily possess infinite energy.
Surprisingly, the question of classical regularity of solutions remains open even in this simplified case.  
%
The fundamental issue of \eqref{VM} persists within \eqref{1.5DVM}, namely that the reduced Vlasov characteristics in the density equation propagate at an uncontrollable speed $\vert v_1 \vert$ and hence, are able to intersect the field characteristics which propagate with speed $c=1$.  Though we cannot currently prove that all initial data launch a global-in-time solution, we can provide an answer for certain classes of initial data (see Section $3$) and establish uniform bounds on derivatives arbitrarily close to this possible intersection of characteristics.

This paper proceeds as follows.  In the next section, we will derive \emph{a priori} estimates in order to prove the main result.  The first lemma obtains bounds on the spatial and velocity supports of the particle distribution and the associated fields, while the main theorem guarantees a uniform bound on derivatives of the distribution function as long as particle velocities do not approach the set on which $\vert v_1 \vert = 1$. The proofs of these results then follow in the latter portion of the section.
Finally, in Section $3$, we present a theorem concerning global existence for initial particle densities that are even in $v_2$.  We show that this symmetry property is preserved in time by solutions of \eqref{1.5DVM} and a reduction in complexity occurs which allows us to conclude that smooth solutions exist globally in time.

Throughout the paper the value $C > 0$ will denote a generic constant that may change from line to line and depend upon the neutralizing density $b$, existence time $T$, and initial data \eqref{IC}.  When necessary, we will specifically identify a constant with a subscript (e.g., $C_1$).
Finally, since we are interested in classical solutions, we will assume $f^0, E_2^0, B^0 \in C_c^1(\bfR^2)$ for the entirety of the paper.

\section{\emph{A priori} Estimates}

To begin this section, we will first prove a lemma that bounds the support of the particle density and the associated electric and magnetic fields.

\begin{lemma}
\label{L1}
Let $T > 0$ be given and assume that $(f, E, B)$ is a $C^1$ solution of (\ref{1.5DVM}) on $[0,T)$ with 
$f_0\geq 0$, and $b \in C^1_c(\bfR)$ satisfying the global neutrality assumption
$$ \iint f^0(x,v) \ dv \ dx = \int b(x) \ dx.$$
Then, there exists $C > 0$, 
such that
$f(t,x,v) \neq 0$ for $ t\in [0,T), x \in \bfR$, and $v \in \bfR^2$ implies
\begin{equation}
\label{suppbound}
\vert x \vert + \vert v \vert \leq C.
\end{equation}
Additionally, there is $C > 0$ such that
$$ \vert E(t,x) \vert + \vert B(t,x) \vert \leq C$$
for all $t \in [0,T), x \in \bfR$.
\end{lemma}

With this result in hand, we may further obtain bounds on derivatives of the particle distribution function.
\begin{theorem}
\label{T1}
Let the assumptions of Lemma \ref{L1} hold. Then, for any $\epsilon > 0$ the quantity $\vert \partial_x f \vert + \vert \nabla_v f \vert$ is uniformly bounded on the set
$$S_\epsilon := \left  \{ (t,x,v) \in [0,T) \times \bfR \times \bfR^2 : \biggl \vert \vert v_1 \vert - 1 \biggr \vert > \epsilon \right \}.$$ 
\end{theorem}

Prior to proving these results, we introduce a few other quantities necessary for the subsequent analysis.  First, we define characteristics for the Vlasov equation.  These are the curves $X(s,t,x,v)$ and $V(s,t,x,v)$ satisfying
$$\left\{
\begin{gathered}
\frac{\partial X}{\partial s}=V_1, \qquad
X(t,t,x,v)=x, \\
\frac{\partial V_1}{\partial s}=E_1(s,X) + V_2 B(s,X), \qquad
V_1(t,t,x,v)=v_1\\
\frac{\partial V_2}{\partial s}=E_2(s,X) - V_1 B(s,X), \qquad
V_2(t,t,x,v)=v_2.\\
\end{gathered} \right.
$$
Often, the $(t,x,v)$ dependence of these curves will be suppressed so, for example, $X(s,t,x,v)$ will be denoted by $X(s)$ for brevity.
Then, the Vlasov equation can be expressed as a derivative along the characteristic curves by $\frac{d}{ds} f(s, X(s), V(s)) = 0$. Thus, we find $$f(t,x,v) = f^0(X(0,t,x,v), V(0,t,x,v)) \in [0,C] .$$

\begin{remark}
If it can be shown that velocity characteristics remain bounded away from the set $\vert V_1(s) \vert = 1$ for all $s \geq 0$, then Theorem \ref{T1} implies the global existence of smooth solutions.
\end{remark}

Next, define the potential
$$A(t,x) = \int_{-\infty}^x B(t,y) \ dy,$$
and notice that 
$$\partial_t A = \int_{-\infty}^x \left ( - \partial_x E_2(t,y) \right ) \ dy = -E_2(t,x)$$
and
\begin{equation}
\label{waveA}
\left (\partial_{tt} - \partial_{xx} \right) A = - \partial_t E_2 - \partial_x B = j_2.
\end{equation}
The potential will be used throughout this section because it satisfies the important identity
$$\frac{d}{ds} \left [ V_2(s) + A(s,X(s)) \right] = 0$$
so that
\begin{equation}
\label{v2A}
v_2 + A(t,x) = V_2(0) + A(0,X(0)).
\end{equation}
Now, we may prove Lemma \ref{L1}.

\begin{proof}
To prove the initial result, we first define a function to serve as an upper bound on the maximal velocity support
$$Q(t) : = 1 + \sup \left \{ \vert v \vert : \mathrm{\ there \ exist \ } \tau \in [0,t], x \in \bfR \mathrm{\ such \ that \ } f(\tau, x, v) \neq 0 \right \}$$
and the blow-up time
$$\bar{T} = \sup \left \{ t \in [0,T): Q(t) \mathrm{ \ is \ finite } \right \}.$$
Then, $\bar{T} > 0$ and $Q:[0,\bar{T}) \to [0,\infty)$ is continuous and nondecreasing.  Finally, if $\bar{T} < T$ then
$Q(t) \to \infty$ as $t \to \bar{T}^-$.

Consider any $T_1 \in [0,\bar{T})$.
For $t \in [0,T_1]$ it is straightforward to check the energy conservation identity
$$\frac{d}{dt} \left ( \iint \vert v \vert^2 f(t,x,v) \ dv \ dx + \int \left ( \vert E(t,x) \vert^2 + B(t,x)^2 \right) \ dx \right ) = 0$$
so that
$$\iint \vert v \vert^2 f(t,x,v) \ dv \ dx + \int \left (\vert E(t,x) \vert^2 + B(t,x)^2 \right ) \ dx = C.$$
Using the supremum bound on $f$ and the kinetic energy portion of the identity, we estimate $j$.  So, for every $P > 0$
\begin{eqnarray*}
\vert j(t,x) \vert & \leq & \int_{\vert v \vert < P} \vert v \vert f \ dv + P^{-1} \int_{\vert v \vert > P} \vert v \vert^2 f \ dv\\
& \leq & CP^3 + P^{-1} \int \vert v \vert^2 f \ dv.
\end{eqnarray*}
Taking $P = \left ( \int \vert v \vert^2 f \ dv \right )^{1/4}$ then yields
\begin{equation}
\label{jest}
\vert j(t,x) \vert \leq C  \left ( \int \vert v \vert^2 f(t,x,v) \ dv \right )^{3/4}.
\end{equation}
Note that \eqref{jest} holds if $j=0$ also.  
By \eqref{waveA}, \eqref{jest}, and H\"{o}lder's inequality we have
\begin{eqnarray*}
\vert A(t,x) \vert & = & \frac{1}{2} \left \vert A(0,x+t) + A(0,x-t) - \int_{x-t}^{x+t} E_2(0,y) \ dy + \int_0^t \int_{x-t+\tau}^{x+t-\tau} j_2(\tau,y) \ dy \ d\tau \right \vert\\
& \leq & C \left ( 1 + \int_0^t \int_{x-t+\tau}^{x+t-\tau} \left (\int \vert v \vert^2 f(\tau,y, v) \ dv \right )^{3/4} \ dy \ d\tau \right)\\
& \leq & C \left ( 1 + \left (\int_0^t \int \int \vert v \vert^2 f(\tau,y, v) dv \ dy \ d\tau  \right )^{3/4} \left (\int_0^t \int_{x-t+\tau}^{x+t-\tau} \ dy \ d\tau \right )^{1/4} \right)\\
& \leq & C.
\end{eqnarray*}

Now, if $f(t,x,v) \neq 0$ then by \eqref{v2A} we find
$$\vert v_2 + A(t,x) \vert = \vert V_2(0,t,x,v) + A(0,X(0,t,x,v)) \vert \leq C$$
so that 
\begin{equation}
\label{v2supp}
\vert v_2 \vert \leq \vert v_2 + A(t,x) \vert + \vert A(t,x) \vert \leq C.
\end{equation}
This bound then yields a spatially-uniform bound on $j_2$
$$\vert j_2(t,x) \vert \leq \int_{-Q(t)}^{Q(t)} \int_{-C}^{C} \vert v_2 \vert f \ dv_2 \ dv_1 \leq CQ(t).$$

Next, we note that the field equations in \eqref{1.5DVM} imply
$$\begin{gathered}
(\partial_t + \partial_x) (E_2 + B) = -j_2\\
(\partial_t - \partial_x) (E_2 - B) = - j_2.
\end{gathered}$$
The first of these equations yields the bound
$$\vert E_2(t,x) + B(t,x) \vert = \left \vert E_2(0,x-t) + B(0,x-t) - \int_0^t j_2(\tau, x-t+\tau) \ d\tau \right \vert \leq CQ(t)$$
and an identical bound for $\vert E_2 - B \vert$ follows from the second equation in the same manner.
Using these together, we find
\begin{equation}
\label{E2Bbound}
\vert E_2(t,x) \vert + \vert B(t,x) \vert \leq C Q(t).
\end{equation}
In addition, a bound on the first component of the electric field arises from charge conservation.
In particular, integrating the Vlasov equation over phase space yields
$$\frac{d}{dt} \iint f(t,x,v) \ dv \ dx = 0$$
and thus
$$\iint f(t,x,v) \ dv \ dx = \iint f^0(x,v) \ dv \ dx.$$
With this, we find
\begin{eqnarray*}
\vert E_1(t,x) \vert & = & \left \vert \int_{-\infty}^x \left ( \int f(t,y,v) \ dv - b(y) \right ) \ dy \right \vert\\
& \leq & \iint f(t,y,v) \ dv \ dy + \int \vert b(y) \vert \ dy\\
& = & \iint f^0(y,v) \ dv \ dy + \int \vert b(y) \vert \ dy\\
& \leq & C.
\end{eqnarray*}
Hence, if $f(t,x,v) \neq 0$, then
$$\left \vert \frac{dV_1}{ds} \right \vert = \left \vert E_1(s, X) + V_2 B(s,X) \right \vert \leq C Q(s)$$
and therefore
$$\vert v_1 \vert \leq \left \vert V_1(0,t,x,v) + \int_0^t \frac{dV_1}{ds} \ ds \right \vert \leq C\left (1 + \int_0^t Q(s) \ ds \right ).$$
Combining this with the bound on the $v_2$-support in \eqref{v2supp}, it follows that
$$Q(t) \leq C\left ( 1 + \int_0^t Q(s) \ ds \right ).$$
By Gronwall's inequality, we conclude
$$Q(t) \leq C$$
and \eqref{suppbound} follows directly from this. Finally, the field bounds follow precisely from this estimate and \eqref{E2Bbound}. Similarly, $\rho$ and $j$ are controlled by the estimate on $Q(t)$.
\end{proof}

To conclude this section, we prove Theorem \ref{T1}.
\begin{proof}
For any $t \in [0,T)$ define the norm
$$\Vert \nabla_{x,v} f(t) \Vert := \sup_{\tau \in [0,t]} \Vert \nabla_{x,v} f(\tau)  \Vert_\infty $$
and note that the function $t \to \Vert \nabla_{x,v} f(t) \Vert$ is continuous and nondecreasing, and maps $[0,T)$ to $[0,\infty)$.
To prove the uniform boundedness asserted in Theorem \ref{T1}, consider $\epsilon > 0$ and without loss of generality take $\epsilon < T/2$.
Let $(t,x,v) \in S_\epsilon \bigcap \mathrm{supp}(f).$
If this intersection is, in fact, empty then the uniform bound on derivatives guaranteed by the theorem will merely be $0$.
If $t \leq T - \epsilon$ then
$$ \vert \partial_x f(t,x,v) \vert + \vert \nabla_v f(t,x,v) \vert \leq \Vert \nabla_{x,v} f(T-\epsilon) \Vert,$$
so consider $t \in (T-\epsilon, T)$ and note that
$$\biggl \vert \vert V_1(t) \vert - 1 \biggr \vert = \biggl \vert \vert v_1 \vert - 1 \biggr \vert > \epsilon.$$
From the result of Lemma \ref{L1}, namely \eqref{suppbound}, \eqref{E2Bbound}, and the bound on $E_1$, it follows that
$$\vert K(t,x,v) \vert \leq C$$
on the support of $f$.
Hence, there is $C_1 > 1$ such that $t_\epsilon : = t - \frac{\epsilon}{C_1} \leq s \leq t$ implies
$$\biggl \vert \vert V_1(s) \vert - 1 \biggr \vert > \frac{1}{2} \epsilon.$$
Thus, we have
\begin{equation}
\label{V1bound}
\int_{t_\epsilon}^t \biggl \vert \vert V_1(s) \vert - 1 \biggr \vert^{-1} ds \leq \frac{2}{\epsilon}(t-t_\epsilon) = \frac{2}{C_1}.
\end{equation}

From the Vlasov equation, we have
$$\frac{d}{ds} \left [\partial_x f(s, X(s), V(s)) \right ] = - \left ( \partial_x K \cdot \nabla_v f \right )(s,X(s), V(s)),$$
\begin{equation}
\label{dv1f}
\frac{d}{ds} \left [\partial_{v_1} f(s, X(s), V(s)) \right ] = \left ( -\partial_x f + B \partial_{v_2} f \right )(s,X(s), V(s)),
\end{equation}
\begin{equation}
\label{dv2f}
\frac{d}{ds} \left [\partial_{v_2} f(s, X(s), V(s)) \right ] = - \left (B \partial_{v_1} f \right )(s,X(s), V(s)).\\
\end{equation}
Hence, for $t_2 \in [t_\epsilon, t]$
\begin{eqnarray*}
\partial_x f(t_2, X(t_2), V(t_2)) & = & \partial_x f(t_\epsilon, X(t_\epsilon), V(t_\epsilon)) - \int_{t_\epsilon}^{t_2} \left ( \partial_x K \cdot \nabla_v f \right )(s,X(s),V(s)) \ ds \\
& = &  \partial_x f(t_\epsilon, X(t_\epsilon), V(t_\epsilon))\\
& \ & \ - \int_{t_\epsilon}^{t_2} \partial_x K(s,X(s),V(s)) \cdot \biggl [ \nabla_v f(t_\epsilon, X(t_\epsilon), V(t_\epsilon))\\
& \ & \  + \int_{t_\epsilon}^s \biggl \langle -\partial_x f + B\partial_{v_2} f, -B\partial_{v_1} f \biggr \rangle (u,X(u),V(u)) du \biggr] \ ds
\end{eqnarray*}
Letting 
\begin{equation}
\label{K}
\mcK(u, t_2) = \int_u^{t_2} \partial_x K(s, X(s), V(s)) \ ds
\end{equation}
and changing the order of integration in the last term, this expression becomes
\begin{equation}
\label{derivs}
\begin{aligned}
\partial_x f(t_2, X(t_2), V(t_2)) & = \partial_x f(t_\epsilon, X(t_\epsilon), V(t_\epsilon))\\
& \  - \mcK(t_\epsilon, t_2) \cdot \nabla_v f(t_\epsilon, X(t_\epsilon), V(t_\epsilon)) \\
& \  - \int_{t_\epsilon}^{t_2} \mcK(u, t_2) \cdot \biggl \langle -\partial_x f + B\partial_{v_2} f, -B\partial_{v_1} f \biggr \rangle (u,X(u),V(u)) du.
\end{aligned}
\end{equation}
Now, it remains to bound $\mcK$.  For $\mcK_2$, note that
\begin{eqnarray*}
\partial_x K_2(s,X(s), V(s)) & = & \partial_x E_2(s,X(s)) - V_1(s)\partial_x B(s,X(s))\\
& = & - \partial_t B(s,X(s)) - V_1(s) \partial_x B(s,X(s))\\
& = & - \frac{d}{ds} \biggl [B(s,X(s)) \biggr],
\end{eqnarray*}
so that
\begin{equation}
\label{K2bound}
\vert \mcK_2(u, t_2) \vert = \biggl \vert B(u, X(u)) - B(t_2, X(t_2)) \biggr \vert \leq C.
\end{equation}
For $\mcK_1$, note that
$$\partial_x K_1 = \partial_x E_1 + v_2 \partial_x B$$
and
\begin{equation}
\label{dxE1}
\vert \partial_x E_1(t,x) \vert = \vert \rho(t,x) \vert  \leq C,
\end{equation}
so we focus on the remaining quantity in \eqref{K}, namely
\begin{equation}
\label{v2dxB}
\int_u^{t_2} V_2(s) \partial_xB(s,X(s)) \ ds.
\end{equation}

Since $B$ can be represented as
$$B(t,x) = \frac{1}{2} \biggl [ (E_2 + B)(0,x-t) - (E_2 - B)(0, x+t) - \int_0^t ( j_2(\tau, x - t + \tau) - j_2(\tau, x + t - \tau)) \ d\tau \biggr ], $$
we'll consider
\begin{eqnarray*}
\mcB^\pm(t,x) & := & \int_0^t \partial_x j_2(\tau, x\mp(t-\tau)) \ d\tau\\
& = & \int_0^t \int v_2 \partial_x f(\tau, x \mp(t-\tau), v) \ dv \ d\tau.
\end{eqnarray*}
Then, we can estimate \eqref{v2dxB} by
\begin{equation}
\label{v2dxBbound}
\left \vert \int_u^{t_2} V_2(s) \partial_x B(s,X(s)) \ ds \right \vert \leq C 
+ \frac{1}{2} \left \vert \int_u^{t_2} V_2(s) [ \mcB^+(s, X(s)) - \mcB^-(s, X(s))] \ ds \right \vert.
\end{equation}
Consider the first part of the expression on the right, namely
$$I: = \int_u^{t_2} V_2(s) \mcB^+(s,X(s)) \ ds,$$
as the other term may be handled similarly.
Writing
$$f = f(\tau, X(s) - s + \tau, v) \qquad \mathrm{and} \qquad \partial_x f = \partial_x f(\tau, X(s) - s + \tau, v)$$
we have from the Vlasov equation
$$\partial_x f = \frac{1}{1- v_1}\biggl [\frac{df}{d\tau} + \nabla_v \cdot (fK) \biggr ] = \frac{1}{V_1(s) - 1} \frac{df}{ds}.$$
Take
$$D = (V_1(s) - 1)^2 + (v_1 - 1)^2 \qquad \mathrm{and} \qquad \theta = \frac{(V_1(s) - 1)^2}{D}$$
and write $\partial_x f$ as a convex combination of its two representations, so that
$$ \partial_x f = \frac{\theta}{V_1(s) - 1} \frac{df}{ds} + \frac{1 - \theta}{1-v_1} \biggl [\frac{df}{d\tau} + \nabla_v \cdot (fK) \biggr ].$$
Using this, we split $I$ into three portions given by 
\begin{equation}
\label{I}
\begin{aligned}
I & = \int_u^{t_2} V_2(s) \int_0^s \int v_2 \partial_x f(\tau, X(s) - s + \tau, v) \ dv \ d\tau \ ds\\
& = \int_u^{t_2} V_2(s) \int_0^s \int v_2 \left [\frac{V_1(s) - 1}{D} \frac{df}{ds} + \frac{1-v_1}{D}\frac{df}{d\tau} + \frac{1-v_1}{D} \nabla_v \cdot (fK) \right ] \ dv \ d\tau \ ds\\
& =:  I_S + I_T + I_V.
\end{aligned}
\end{equation}
Note that
\begin{equation}
\label{D}
\int_{\vert v \vert \leq C} \frac{1}{D} \ dv \leq C\int \frac{dv_1}{(V_1(s) - 1)^2 + (v_1 - 1)^2} = \frac{C}{\vert V_1(s) - 1\vert}.
\end{equation}
Using \eqref{D} we have
\begin{equation}
\label{IT}
\begin{aligned}
\vert I_T \vert & = \left \vert \int_u^{t_2} V_2(s) \int v_2 \frac{1-v_1}{D} \left ( f(s,X(s), v) - f(0, X(s) - s, v) \right ) \ dv \ ds \right \vert\\
& \leq C \int_u^{t_2} \int_{\vert v \vert \leq C} \frac{dv \ ds}{D} \leq C \int_u^{t_2} \frac{ds}{\vert V_1(s) - 1 \vert}.
\end{aligned}
\end{equation}

Next, we consider $I_V$.  
Using \eqref{D} again and integrating by parts, we have for the portion involving $\partial_{v_2}$,
\begin{eqnarray*}
\left \vert \int v_2 \frac{1-v_1}{D} \partial_{v_2} (f K_2) \ dv \right \vert & = & \left \vert -\int fK_2 \partial_{v_2} \left ( v_2 \frac{1-v_1}{D} \right ) \ dv \right \vert\\
& = &  \left \vert \int fK_2 \frac{1-v_1}{D}\ dv \right \vert\\
& \leq & C \int_{\vert v \vert \leq C} \frac{1}{D} \ dv\\
& \leq & \frac{C}{\vert V_1(s) - 1 \vert} .
\end{eqnarray*}
For the portion involving $\partial_{v_1}$, we again use \eqref{D} to find
\begin{eqnarray*}
\left \vert \int v_2 \frac{1-v_1}{D} \partial_{v_1} (f K_1) \ dv \right \vert & = & \left \vert -\int fK_1 \partial_{v_1} \left ( v_2 \frac{1-v_1}{D} \right ) \ dv \right \vert\\
& = & \left \vert \int fK_1 v_2 \left (\frac{1}{D} + \frac{1-v_1}{D^2} 2(v_1 - 1) \right )\ dv \right \vert\\
& \leq & \int \vert v_2 \vert f \vert K_1 \vert \left (\frac{1}{D} + \frac{2D}{D^2} \right ) \ dv\\
& \leq & \frac{C}{\vert V_1(s) - 1 \vert}.
\end{eqnarray*}
Combining these estimates we have
\begin{equation}
\label{IV}
\vert I_V \vert \leq \int_u^{t_2} \vert V_2(s) \vert \int_0^s \frac{C}{\vert V_1(s) - 1 \vert} \ d\tau \ ds \leq C\int_u^{t_2} \vert V_1(s) - 1 \vert^{-1} \ ds.
\end{equation}

For $I_S$, more work is needed.
We integrate by parts twice so that
\begin{eqnarray*}
I_S & = & \int_u^{t_2} V_2(s) \biggl [ \frac{d}{ds} \int_0^s \int v_2 \frac{V_1(s) - 1}{D} f \ dv \ d\tau\\
& \ & - \int v_2 \frac{V_1(s) - 1}{D} f(s, X(s), v) \ dv \\
& \ & - \int_0^s \int v_2 \frac{d}{ds} \left ( \frac{V_1(s) - 1}{D} \right ) f \ dv \ d\tau \biggr ] \ ds \\
& = & V_2(s) \int_0^s \int v_2 \frac{V_1(s) - 1}{D} f \ dv \ d\tau \biggr \vert_{s = u}^{s = t_2} \\
& \ & - \int_u^{t_2} \frac{dV_2}{ds} \int_0^s \int v_2 \frac{V_1(s) - 1}{D} f \ dv \ d\tau \ ds \\
& \ & - \int_u^{t_2} V_2(s) \int v_2 \frac{V_1(s) - 1}{D} f(s,X(s), v) \ dv \ ds\\
& \ & - \int_u^{t_2}  V_2(s) \int_0^s \int v_2 \left ( \frac{1}{D} - \frac{V_1(s) - 1}{D^2}2(V_1(s) - 1) \right ) \frac{dV_1}{ds} f \ dv \ d\tau \ ds \\
& =: & I_S^1 + I_S^2 + I_S^3+ I_S^4.
\end{eqnarray*}
Note that
$$\int_{\vert v\vert \leq C} \frac{\vert V_1(s) - 1 \vert}{D} \ dv \leq C\int \frac{\vert V_1(s) - 1 \vert}{(v_1 - 1)^2 + \vert V_1(s) - 1 \vert^2} \ dv_1 \leq C$$
and it follows that
\begin{equation}
\label{IS123}
\vert I_S^1 \vert + \vert I_S^2 \vert + \vert I_S^3 \vert \leq C.
\end{equation}
Also, by \eqref{D} we have
\begin{equation}
\label{IS4}
\begin{aligned}
\vert I_S^4 \vert & \leq C\int_u^{t_2} \vert V_2(s) \vert \int_0^s \int_{\vert v \vert \leq C} \vert v_2 \vert \left ( \frac{1}{D} + \frac{2D}{D^2} \right ) \ dv \ d\tau \ ds\\
& \leq C\int_u^{t_2} \vert V_1(s) - 1 \vert^{-1} \ ds.
\end{aligned}
\end{equation}
Using \eqref{IT}, \eqref{IV}, \eqref{IS123}, and \eqref{IS4} in \eqref{I} yields
$$\vert I \vert \leq C \left ( 1 + \int_u^{t_2} \vert V_1(s) - 1 \vert^{-1} \ ds \right ).$$
Similar steps show that
$$\left \vert \int_u^{t_2} V_2(s) \mcB^-(s, X(s)) \ ds \right \vert \leq C \left ( 1 + \int_u^{t_2} \vert V_1(s) + 1 \vert^{-1} \ ds \right )$$
and hence \eqref{v2dxBbound} yields
$$\left \vert \int_u^{t_2} V_2(s) \partial_x B(s,X(s)) \ ds \right \vert \leq C \left ( 1 + \int_u^{t_2} \left [ \vert V_1(s) - 1 \vert^{-1} + \vert V_1(s) + 1 \vert^{-1} \right ] \ ds \right ).$$
By \eqref{V1bound} we have
$$\left \vert \int_u^{t_2} V_2(s) \partial_x B(s,X(s)) \ ds \right \vert \leq C.$$
Using this bound, \eqref{K2bound}, and \eqref{dxE1} within \eqref{K} yields
$$\vert \mcK(u, t_2) \vert \leq C.$$

Returning to \eqref{derivs}, we find
$$\begin{aligned}
\left \vert \partial_x f(t_2, X(t_2), V(t_2)) \right \vert & \leq \left \vert \partial_x f(t_\epsilon, X(t_\epsilon), V(t_\epsilon)) \right \vert \\
& \ + C \left \vert \nabla_v f(t_\epsilon, X(t_\epsilon), V(t_\epsilon)) \right \vert \\
& \  + C\int_{t_\epsilon}^{t_2} \left ( \vert \partial_x f \vert + \vert \nabla_v f \vert \ \right ) (u,X(u),V(u)) du\\
& \leq C \left ( \Vert \nabla_{x,v} f(t_\epsilon) \Vert + \int_{t_\epsilon}^{t_2} \left ( \vert \partial_x f \vert + \vert \nabla_v f \vert \ \right ) (u,X(u),V(u)) du \right ).
\end{aligned}$$
From \eqref{dv1f} and \eqref{dv2f} and the previous bounds it follows that
$$\left \vert \nabla_{x,v} f(t_2, X(t_2), V(t_2)) \right \vert \leq C\Vert \nabla_{x,v} f(t_\epsilon) \Vert
+ C \int_{t_\epsilon}^{t_2} \left \vert \nabla_{x,v} f (u,X(u),V(u)) \right \vert du.$$
By Gronwall's inequality, we finally have
\begin{eqnarray*}
\left \vert \nabla_{x,v} f(t, X(t), V(t)) \right \vert & \leq & C \Vert \nabla_{x,v} f(t_\epsilon) \Vert e^{C(t- t_\epsilon)}\\
& \leq & C \Vert \nabla_{x,v} f(t_\epsilon) \Vert e^{CT}\\
& \leq & C \Vert \nabla_{x,v} f(t_\epsilon) \Vert\\
& \leq & C \Vert \nabla_{x,v} f(T -\epsilon/C_1) \Vert.
\end{eqnarray*}
From this, the conclusion of Theorem \ref{T1} follows, and the proof is complete.
\end{proof}

\section{Global existence for symmetric initial data}

Though it remains unknown as to whether any smooth triple of initial data $(f^0, E_2^0, B^0)$ launches a global-in-time classical solution $(f,E_2,B)$ of \eqref{1.5DVM}, we may show that a particular class of solutions exists globally in time.
In particular, we will show that symmetry of the initial distribution in $v_2$ is preserved in time and gives rise to a smooth, unique, global solution, which satisfies the one-dimensional Vlasov-Poisson system.
Throughout this final section we will split the $v$-dependence of the distribution function into its components $v_1$ and $v_2$ for clarity.
  
\begin{theorem}
Assume the initial particle distribution $f^0(x,v_1,v_2)$ satisfies $$f^0(x, v_1, v_2) = f^0(x, v_1, -v_2)$$ for every $x, v_1, v_2 \in \bfR$, while the initial fields satisfy $E_2^0(x) = B^0(x) = 0$ for every $x \in \bfR$. Then, there is a unique $f \in C^1\left ([0,\infty) \times \bfR^3 \right )$ satisfying \eqref{1.5DVM} and $$f(t,x,v_1,v_2) = f(t,x,v_1,-v_2)$$ for all $t \in [0,\infty)$, $x,v_1,v_2 \in \bfR$.
In particular, $E_2(t,x) = B(t,x) = 0$ for all $t \in [0,\infty)$, $x \in \bfR$ and $f$ satisfies the one-dimensional Vlasov-Poisson system in the variables $(t,x,v_1)$ with $v_2$ as a parameter, namely
$$ \left \{ \begin{gathered}
 \partial_t f + v_1 \partial_x f + E_1 \partial_{v_1} f = 0.\\
\partial_x E_1 = \rho(t,x) = \int f dv - b(x), \qquad \partial_t E_1 = - j_1 = - \int v_1 f \ dv. \\
\end{gathered} \right.$$ 
\end{theorem}

\begin{proof}
Let $g(t,x,v_1,v_2) = f(t, x,v_1,-v_2)$ for all $t \in [0,T)$, $x, v_1, v_2 \in \bfR$ so that 
\begin{equation}
\label{ICg}
g(0,x,v_1,v_2) = f^0(x,v_1,-v_2) = f^0(x,v_1,v_2).
\end{equation}
Then, define $j_2^g(t,x) = \int w_2 g(t,x,w_1,w_2) \ dw_2 dw_1$ and note that, upon changing variables
\begin{eqnarray*}
j_2(t,x) & = & \int v_2 f(t,x,v_1,v_2) \ dv\\
& = & \iint v_2 g(t,x,v_1,-v_2) \ dv_2 \ dv_1\\
& = & -\iint w_2 g(t,x,w_1,w_2) \ dw_2 \ dw_1\\
& = & - j_2^g(t,x).
\end{eqnarray*}

Let $E_2^g(0,x) = B^g(0,x) = 0$, and for $t \in (0,T)$ define $E_2^g(t,x)$ and $B^g(t,x)$ as the unique solutions of
\begin{equation}
\label{E2B}
\left \{ \begin{gathered}
\partial_t E_2^g + \partial_x B^g = -j_2^g\\
\partial_t B^g + \partial_x E_2^g = 0.
\end{gathered} \right.
\end{equation}
Then, we see that these functions satisfy
$$ \left \{ \begin{gathered}
\partial_t(-E_2^g) + \partial_x(-B^g) = j_2^g = - j_2\\
\partial_t(-B^g) + \partial_x(-E_2^g) = 0.
\end{gathered} \right. $$
Hence, the pair $(-E_2^g, -B^g)$ satisfies the same system of PDEs with the same initial conditions as the pair $(E_2, B)$.  By uniqueness $E_2^g(t,x) = -E_2(t,x)$ and $B^g(t,x) = -B(t,x)$ for all $t \in [0,T), x \in \bfR$.
Since a change of variable in $v_2$ does not affect other quantities in the system, we define $\rho^g, j_1^g$, and $E_1^g$, and note that they are equal to their $f$-dependent counterparts, so that
$$\begin{gathered}
\rho^g(t,x) := \int g(t,x,v_1,v_2) \ dv_2 dv_1 - b(x)  = \rho(t,x)\\
j_1^g(t,x) := \int v_1 g(t,x,v_1,v_2) \ dv_2 dv_1 = j_1(t,x)\\
\partial_x E_1^g(t,x) := \rho^g(t,x) = \rho(t,x).
\end{gathered}$$
Then, we apply the Vlasov operator in the variables $(t,x,w_1,w_2)$ to the function $g(t,x,w_1,w_2)$ 
to find
$$\begin{gathered}
\partial_t g + w_1 \partial_x g + (E_1^g + w_2 B^g)\partial_{w_1} g + (E_2^g - w_1 B^g) \partial_{w_2} g\\ 
= \partial_t f + w_1 \partial_x f + (E_1 - w_2 B )\partial_{v_1} f - (-E_2 + w_1 B) \partial_{v_2} f 
\end{gathered}$$
where the function $f=f(t,x,v_1,v_2)$ is evaluated at the point $(v_1, v_2) = (w_1,-w_2)$.
Relabeling the velocity arguments in this equation using $v_1 = w_1$ and $v_2 = -w_2$, we find further
$$\begin{gathered}
\partial_t g + w_1 \partial_x g + (E_1^g + w_2 B^g)\partial_{w_1} g + (E_2^g - w_1 B^g) \partial_{w_2} g\\ 
= \partial_t f + v_1 \partial_x f + (E_1 + v_2 B )\partial_{v_1} f + (E_2 - v_1 B) \partial_{v_2} f \hfill \\
= 0. \hfill
\end{gathered}$$
Hence, $g$ satisfies the Vlasov equation with initial condition \eqref{ICg}.  Additionally, $E_2^g$ and $B^g$ satisfy the analogous transport equations as $E_2$ and $B$, namely \eqref{E2B}, with the same initial conditions $E_2^g(0,x) = 0 = E_2(0,x)$ and $B^g(0,x) = 0 = B(0,x)$.
By uniqueness, we find $(g,E_2^g,B^g) \equiv (f,E_2,B)$, which in particular implies
$$f(t,x,v_1,-v_2) = f(t,x,v_1,v_2)$$
for every $t \in [0,T)$, $x, v_1, v_2 \in \bfR$.
Since $f$ is even in $v_2$, the function $v_2 f$ is odd and we find $j_2(t,x) = \iint v_2 f(t,x,v_1, v_2) \ dv_2 dv_1 = 0$.  Additionally, we have $E_2^g(t,x) = E_2(t,x)$ and $E_2^g(t,x) = -E_2(t,x)$, which implies $E_2 \equiv 0$.
As the same equalities hold for $B$, we conclude $B \equiv 0$ as well.
Finally, using these field representation within the Vlasov equation, we see that $f$ satisfies the reduced equation
$$ \partial_t f + v_1 \partial_x f + E_1 \partial_{v_1} f = 0.$$
We note that only the last term on the left side is nonlinear, and hence the coupling between the unknown field $E_1$ and the particle distribution only occurs via the remaining equations
$$ \partial_x E_1 = \rho \qquad \mathrm{\ and \ } \qquad \partial_t E_1 = - j_1.$$  
The resulting system is exactly the one-dimensional Vlasov-Poisson system 
with $v_2$ as a parameter rather than an independent variable. As this system is known to possess a global classical solution, the conclusion of the theorem follows.

\end{proof}


\bibliographystyle{acm}

\end{document}